\documentclass{amsart}%
\usepackage{amsmath}
\usepackage{amsfonts}
\usepackage{amssymb}
\usepackage{amsmath}
\usepackage{amsfonts}
\usepackage{amssymb}
\usepackage{eufrak}
\usepackage{amscd}
\usepackage{amsthm}
\usepackage{epsfig}
\usepackage{amstext}
\usepackage[all]{xy}
\usepackage{graphicx}%
\providecommand{\U}[1]{\protect \rule{.1in}{.1in}}
\newtheorem{theorem}{Theorem}

\newtheorem{corollary}[theorem]{Corollary}

\newtheorem{definition}[theorem]{Definition}

\newtheorem{lemma}[theorem]{Lemma}

\newtheorem{proposition}[theorem]{Proposition}
\newtheorem{remark}[theorem]{Remark}

\begin{document}

\title{Variations of projectivity for $C^*$-algebras}
\author{Don Hadwin}
\address{University of New Hampshire}
\email{operatorguy@gmail.com}
\author{Tatiana Shulman}
\address{Institute of Mathematics of the Polish Academy of Sciences, Poland }
\email{tshulman@impan.pl}
\maketitle

\begin{abstract} We consider various lifting problems for $C^*$-algebras. As an application of our results we show that any commuting family of
order zero maps from matrices to a von Neumann central sequence algebra can be lifted to a commuting family of order zero maps to the $C^*$-central sequence algebra.
\end{abstract}

\bigskip

\section{Introduction}
Many important properties of $C^*$-algebras are formulated in terms of liftings.

By a lifting property we mean the following. Suppose we are given a surjective $\ast$-homomorphism   $\mathcal B  \twoheadrightarrow \mathcal M$. We will say that a C*-algebra $\mathcal A$ has the lifting property corresponding to this surjection if for any $\ast$-homomorphism $\phi: \mathcal A \to \mathcal M$ there is a $\ast$-homomorphism $\psi: \mathcal A \to \mathcal B$ such that the following diagram commutes.

 $$\xymatrix { & \mathcal B \ar@{->>}[d]^-{}  \\
\mathcal A \ar[r]_-{\phi}  \ar@{-->}[ur]^{\psi} & \mathcal M} $$

 In other words any $\ast$-homomorphism from $\mathcal A$  to the $C^*$-algebra $\mathcal M$ "downstairs" "lifts" to a $\ast$-homomorphism to the $C^*$-algebra $\mathcal B$ "upstairs".

$C^*$-algebras which have the lifting property with respect to any surjection are called {\it projective}, they were introduced by B. Blackadar in \cite{Bl}.

Many problems that arise  in C*-algebras reduce to the question of the existence of liftings in various special situations.
Here are some examples:

1) Problems about approximation of almost commuting matrices by commuting ones and, more generally, matricial weak semiprojectivity for $C^*$-algebras (\cite{LoringBook}, \cite{Lin}, \cite{FR}, \cite{ELP}), is expressed as the lifting property corresponding to the surjection $\prod_{n\in \mathbb N} M_{n} \twoheadrightarrow \prod_{n\in \mathbb N} M_{n}/\oplus_{n\in \mathbb N}M_n$ (here $M_n$ is the $C^*$-algebra of all n-by-n matrices).

2) Stability of $C^*$-algebraic relations under small Hilbert-Schmidt perturbations in matrices  is expressed as the lifting property corresponding to the surjection $\prod_{n\in \mathbb N} M_n \twoheadrightarrow \prod_{n\in \mathbb N}^{\alpha} (M_n, tr_n)$ (here $\alpha$ is a non-trivial ultrafilter on $\mathbb N$ and the $C^*$-algebra $\prod_{n\in \mathbb N}^{\alpha} (M_n, tr_n)$ "downstairs" is the tracial ultraproduct of the matrix algebras) \cite{TracialStability}. Stability under small tracial perturbations in $II_1$-factors is expressed as the lifting property corresponding to the surjection $\prod_{n\in \mathbb N} N_n \twoheadrightarrow \prod_{n\in \mathbb N}^{\alpha} (N_n, \tau_n)$ (here $N_n$ is a $II_1$-factor with a faithful trace $\tau_n$) (\cite{TracialStability}).  Similar problems for groups are discussed in \cite{DT2} and \cite{AP}.

3) The property of a $C^*$-algebra to be residually finite-dimensional (RFD)  was proved in \cite{DonRFD}  to be the lifting property corresponding to the surjection $\mathcal B  \twoheadrightarrow B(H)$, where  $\mathcal B \subseteq \prod M_n$ is defined as the $C^*$-algebra of all $\ast$-strongly convergent sequences of matrices and the surjection $\mathcal B  \twoheadrightarrow B(H)$ is defined by sending each sequence to its $\ast$-strong limit. Here we identify $M_n$ with $B(l^2\{1, \ldots, n\}) $ naturally included in $ B(l^2\{\mathbb N\}) = B(H)$.

4) The famous Brown-Douglas-Fillmore theory deals with lifting of
injective $\ast$-homomorphisms from $C(X)$ to the Calkin algebra $\mathcal{C}(H)$ with respect to the surjection $B(H)\twoheadrightarrow   \mathcal{C}(H)$.

5) In the classification program for $C^*$-algebras one sometimes has to deal with liftings of $\ast$-homomorphisms to a von Neumann central sequence algebra $N^{\omega}\cap
N^{\prime}$    to $\ast$-homomorphisms to the $C^*$-central sequence algebra $A_{\omega}\cap A^{\prime}$ (see for instance \cite{TWW}). More details on this and on the surjection $A_{\omega}\cap A^{\prime} \twoheadrightarrow N^{\omega}\cap
N^{\prime}$ are given in section 3.

We see that in the examples above the corresponding surjections sometimes have a von Neumann algebra "upstairs", sometimes "downstairs", sometimes at both places. This leads us to introducing the following more general notions.

 We say that a $C^*$-algebra $\it A$ is {\it $C^*$-$W^*$-projective} if it has the lifting property with respect to any surjection $\mathcal B  \twoheadrightarrow \mathcal M$ with $\mathcal M$ being a von Neumann algebra; in a similar way  {\it $W^*$-$C^*$-projectivity} and {\it $W^*$-$W^*$-projectivity} are defined.
In this terminology the usual projectivity may be called $C^*$-$C^*$-projectivity.

Dealing with specific lifting problems, one has to look at liftability of projections, isometries, matrix units,  various commutational relations, etc. So it is natural to explore whether and  which of those basic relations have more general property of being $C^*$-$W^*$, $W^*$-$W^*$, $W^*$-$C^*$-projective, and  we do it in this paper. Main focus is given to commutational relations, that is to the $C^*$-$W^*$, $W^*$-$W^*$ and  $W^*$-$C^*$-projectivity of commutative $C^*$-algebras, but we consider basic non-commutative relations here as well. Note that for the usual projectivity a characterization of when a separable commutative $C^*$-algebra is projective is obtained in \cite{ChDr} and is the following: $C(K)$ is projective if and only if $K$ is a compact absolute retract of covering dimension not larger than 1.

In section 2 we give necessary definitions and discuss a relation between unital and non-unital cases.

In section 3 we study $C^*$-$W^*$-projectivity. The main result of the section is a characterization of when a separable unital commutative $C^*$-algebra is $C^*$-$W^*$-projective: $C(K)$ is $C^*$-$W^*$-projective if and only if $K$ is connected and locally path-connected  (Theorem  \ref{C*-W*Comm}).
Thus for commutative $C^*$-algebras $C^*$-$W^*$-projectivity is very different from the usual projectivity.
We also give restrictions on a $C^*$-algebra to be $C^*$-$W^*$-projective, namely  it has to be RFD and cannot have non-trivial projections (Propositions \ref{NoProjections} and \ref{RFD}); furthermore we prove that tensoring a separable non-unital commutative $C^*$-$W^*$-projective $C^*$-algebra with matrices preserves $C^*$-$W^*$-projectivity (Theorem \ref{MatricesOverCommAlgebras}).  These results are applied to certain lifting problems for order zero maps (completely positive maps preserving orthogonality).
A commonly used tool in classification of $C^{*}$-algebras is the fact that an
order zero map from the matrix algebra $M_{n}$ to any quotient $C^{*}$-algebra
lifts (the so-called projectivity of order zero maps). In particular a
possibility to lift an order zero map from $M_{n}$ to a von Neumann central sequence algebra $N^{\omega}\cap N^{\prime}$ to
an order zero map to the $C^*$-central sequence algebra $A_{\omega}\cap A^{\prime}$ is a key ingredient to
obtain uniformly tracially large order zero maps (\cite{TWW}). As an application of our results we prove
a stronger statement: one can lift any commuting family of order zero maps
$M_{n} \to N^{\omega}\cap N^{\prime}$ to a commuting family of order zero maps $M_{n}\to A_{\omega}\cap A^{\prime}$ (Theorem \ref{OrderZero}).

In section 4 we study $W^*$-$C^*$-projectivity. This seems to be the most
intractable case. We don't have a characterization  of $W^*$-$C^*$-projectivity for commutative $C^*$-algebras, we only have a sufficient condition (Corollary \ref{totallydisconnected}) and, in case when the spectrum is a Peano continuum, a necessary condition (Proposition \ref{Peano}). We prove basic non-commutative results such as lifting projections and partial isometries,  consider  $W^*$-$C^*$-projectivity of matrix algebras, Toeplitz algebra, Cuntz algebras and discuss a relation with extension groups Ext.  Techniques developed in this section are applied in section 5.

In section 5 we study $W^*$-$W^*$-projectivity. The main result here is that all separable subhomogeneous $C^*$-algebras are $W^*$-$W^*$-projective (Theorem \ref{subhomogeneous}). In particular all separable commutative $C^*$-algebras are $W^*$-$W^*$-projective. We discuss also a relation between $W^*$-$W^*$-projectivity and property RFD. It is easy to show that if a $C^*$-algebra $\mathcal{A}$ is  separable nuclear $W^*$-$W^*$- projective and has a faithful trace, then it must be RFD. Moreover if Connes' embedding problem has an affirmative answer, then every unital
$W^*$-$W^*$-projective C*-algebra with a faithful trace is RFD. The converse to this statement is not true. Indeed in
\cite{TracialStability} we constructed a nuclear  $C^{*}$-algebra which is RFD (hence has a faithful trace) but is
not matricially tracially stable (that is not stable under small Hilbert-Schmidt perturbations in matrices) and hence is not $W^*$-$W^*$-projective.
In this paper we give an example which is not only nuclear but even AF (Theorem \ref{AF}). Our arguments
of why it is not matricially tracially stable are much simpler than the ones
in \cite{TracialStability}.

\bigskip
{\bf Acknowledgements} The first author was supported by a
Collaboration Grant from the Simons Foundation. The research of the
second author was supported by the Polish National Science Centre grant
under the contract number DEC- 2012/06/A/ST1/00256, by the grant H2020-MSCA-RISE-2015-691246-QUANTUM DYNAMICS" and Polish Government grant 3542/H2020/2016/2, and from the Eric Nordgren
Research Fellowship Fund at the University of New Hampshire.

\section{Definitions}

\begin{definition}
Suppose $\mathcal{X}$ and $\mathcal{Y}$ are classes of unital $C^*$-algebras that are
closed under isomorphism. We say that a unital $C^*$-algebra $\mathcal{A}$ is $\mathcal{X}%
$\emph{-}$\mathcal{Y}$\emph{ projective}  if, for every $\mathcal{B}\in \mathcal{X}$,  $\mathcal{M}%
\in \mathcal{Y}$ and unital surjective $\ast$-homomorphisms $\pi:\mathcal{B}%
\rightarrow \mathcal{M}$ and  every unital $\ast$-homomorphism  $\phi:\mathcal{A}\rightarrow \mathcal{M}$, there is a unital $\ast$-homomorphism  $\psi:\mathcal{A}\rightarrow \mathcal{B}$
such that $\pi\circ \psi = \phi.$
\end{definition}

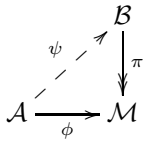
$$\xymatrix { & \mathcal B \ar@{->>}[d]^-{\pi}  \\
\mathcal A \ar[r]_-{\phi}  \ar@{-->}[ur]^{\psi} & \mathcal M} $$

The same conditions with all the words "unital" taken away define $\mathcal{X}%
$\emph{-}$\mathcal{Y}$\emph{ projectivity in the non-unital category}.

 We use the term \emph{$C^*$-$W^*$-projective} when $\mathcal{X}$ is the class of all unital $C^*$-algebras and
$\mathcal{Y}$ is the class of all von Neumann algebras. We use the term \emph{$C^*$-$W^*$-projective in the non-unital category} when $\mathcal{X}$ is the class of all $C^*$-algebras and
$\mathcal{Y}$ is the class of all von Neumann algebras.

The terms
 \emph{W*}-\emph{C*}\emph{-projective(in the non-unital category)},  \emph{W*-W*-projective (in the non-unital category)}, \emph{C*-C*-projective (in the non-unital category)},
 are defined similarly. The usual notion of projectivity defined by B. Blackadar \cite{Bl} is the $C^*$-$C^*$-projectivity in the non-unital category.

  The term $RR0$\emph{-projectivity} is used when
$\mathcal{X}=\mathcal{Y}$ is the class of unital real rank zero $C^*$-algebras.

 Thus  in the introduction in the formulation of some of our results  we in fact should of added "in the non-unital category", which we did not do to not confuse the readers too much.

We will work mostly with the unital category, but with some exceptions. Namely in section 3 dealing with order zero maps one has to consider the non-unital case, and in sections 4 and 5
proving stability of the class of $W^*$-$W^*$ and $W^*$-$C^*$-projective $C^*$-algebras under tensoring with matrices and taking direct sums, one has to deal with the non-unital category.

In fact the relation between the unital and non-unital cases is simple. For a $C^*$-algebra $\mathcal A$, let $\tilde{ \mathcal A} = \mathcal A^+$ if $\mathcal A$ is non-unital and  $\tilde{ \mathcal A} = \mathcal A \oplus \mathbb C$ if $\mathcal A$ is unital.

\begin{proposition} Let $\mathcal A$ be a $C^*$-algebra. Then $\mathcal A$ is W*-C*-projective in the non-unital category (W*-W*, C*-W*, C*-C*-projective in the non-unital category respectively) if and only if
$\tilde{\mathcal A}$ is W*-C*-projective (W*-W*, C*-W*, C*-C*-projective respectively).
\end{proposition}

\section{$C^*$-$W^*$-projectivity}

The following result puts a severe restriction on being C*-W* projective. In
particular, if $C\left(  K\right)  $ is C*-W* projective, then $K$ must be connected.

\begin{proposition}
\label{NoProjections} Let $\mathcal{A}$ be a unital $C^{*}$-algebra. If
$\mathcal{A}$ is C*-W* projective, then $\mathcal{A}$ is $\ast$-isomorphic to
a unital C*-subalgebra of the unitization of the cone of $\mathcal{A}^{**}$.
In particular, $\mathcal{A}$ has no non-trivial projections.
\end{proposition}

\begin{proof}
Let $\mathcal{B}$ be the unitization of the cone of $\mathcal{A}^{**}$. We
know that there is a unital $\ast$-homomorphism $\pi:\mathcal{B\rightarrow
A}^{**}$ and there is a faithful unital $\ast$-homomorphism $\rho
:\mathcal{A}\rightarrow \mathcal{A}^{**}$. If $\mathcal{A}$ is C*-W*
projective, then there must be a unital $\ast$-homomorphism $\tau
:\mathcal{A}\rightarrow \mathcal{B}$ such that $\rho=\pi \circ \tau$. Since
$\rho$ is faithful, $\tau$ is an embedding. However, $\mathcal{B}$ has no
nontrivial projections, so $\mathcal{A}$ has no nontrivial projections.
\end{proof}

\begin{proposition}\label{RFD}
Let $\mathcal{A}$ be a separable $C^{*}$-algebra. If $\mathcal{A}$ is
$C^{*}-W^{*}$-projective in either unital or non-unital category, then
$\mathcal{A}$ is RFD.
\end{proposition}

\begin{proof} Let $H = l^2(\mathbb N)$. We will identify the algebra $M_n$ of $n$-by-$n$ matrices with $$B(l^2\{1, \ldots, n\}) \subseteq B(H).$$ Let $\mathcal B \subseteq \prod M_n$ be the $C^*$-algebra of all $\ast$-strongly convergent sequences and  let $I$ be the ideal of all sequences $\ast$-strongly convergent to zero. Then one can identify $\mathcal B /I$ with $B(H)$ by sending each sequence to its $\ast$-strong limit.
In \cite{DonRFD} the first-named author answered a question of Loring by proving the following: a separable $C^*$-algebra $\mathcal A$ is RFD if and only if  each $\ast$-homomorhism from $\mathcal A$ to $\mathcal B /I$ lifts to a $\ast$-homomorphism from $\mathcal A$ to $\mathcal B$. Since $\mathcal B/I = B(H)$ is a von Neumann algebra, the result follows.
\end{proof}

We next characterize $C^*$-$W^*$- projectivity and $C^*$-$AW^*$-projectivity for separable commutative $C^{*}$-algebras.
Recall that a $C^{\ast}$-algebra is $AW^{\ast}$ \cite{K} if every set of
projections has a least upper bound and every maximal abelian selfadjoint
subalgebra is the C*-algebra generated by its projections.

\begin{theorem}\label{C*-W*Comm}
Let $K$ be a compact metric space. Then the following are equivalent:

\begin{enumerate}
\item $C\left(  K\right)  $ is C*-W* projective,

\item $C\left(  K\right)  $ is C*-AW* projective,

\item $C\left(  K\right)  $ is C*-$\mathcal{Y}$ projective, where
$\mathcal{Y}$ is the class of all unital C*-algebras in which every
commutative separable C*-subalgebra is contained in a commutative C*-algebra
generated by projections.

\item $K$ is a continuous image of $\left[  0,1\right]  $,

\item $K$ is connected and locally path-connected.

\item Every continuous function from a closed subset of $\left[  0,1\right]  $
into $K$ can be extended to a continuous function from $\left[  0,1\right]  $
into $K$.

\item Every continuous function from a closed subset of the Cantor set into
$K$ can be extended to a continuous function from $\left[  0,1\right]  $ into
$K$.
\end{enumerate}
\end{theorem}

\begin{proof}
The implications $\left(  3\right)  \Rightarrow \left(  2\right)
\Rightarrow \left(  1\right)  $ and $\left(  6\right)  \Rightarrow \left(
7\right)  $ are clear.

$\left(  4\right)  \Leftrightarrow \left(  5\right)  $. This is the
Hahn-Mazurkiewicz theorem (\cite{TopologyBook}, Th. 3-30).

$\left(  1\right)  \Rightarrow \left(  4\right)  .$ Suppose $C\left(  K\right)
$ is C*-W* projective. Then there is a compact Hausdorff space $X$ such that
$C\left(  K\right)  ^{**}=C\left(  X\right)  $. We know that there is an
embedding $\varphi:X\rightarrow%
{\displaystyle \prod_{i\in I}}
[0, 1]  $ (product topology). Hence there is a surjective $\ast
$-homomorphism from $C\left(
{\displaystyle \prod_{i\in I}}
[0, 1]  \right)  $ onto $C\left(  K\right)  ^{**}$. Since the
canonical embedding from $C\left(  K\right)  $ to $C\left(  K\right)  ^{**}$
is an injective $\ast$-homomorphism, we know from that $\left(  1\right)  $
there is an injective unital $\ast$-homomorphism from $C\left(  K\right)  $ to
$C\left(
{\displaystyle \prod_{i\in I}}
[0, 1]  \right)  $. Hence there is a continuous surjective map
$\beta:%
{\displaystyle \prod_{i\in I}}
[0, 1]  \rightarrow K$. Suppose $D$ is a countable dense subset of
$K$ and $E$ is the set of all elements of $%
{\displaystyle \prod_{i\in I}}
[0, 1]  $ with finite support, i.e., only finitely many nonzero
coordinates. If $x\in D,$ then there is a countable set $E_{x}$ of $E$ such
that $x\in \overline{\beta \left(  E_{x}\right)  }$. Hence
\[
\beta \overline{\left(  \cup_{x\in D}E_{x}\right)  }=K,
\]
since $\beta \overline{\left(  \cup_{x\in D}E_{x}\right)  }$ is compact and
contains $D$. It follows that there is a countable subset $J\subseteq I$ such
that%
\[
\overline{\left(  \cup_{x\in D}E_{x}\right)  }\subseteq%
{\displaystyle \prod_{i\in J}}
[0, 1]  \times%
{\displaystyle \prod_{i\in I\backslash J}}
\left \{  0\right \}  ;
\]
whence,
\[
\beta \left(
{\displaystyle \prod_{i\in J}}
[0, 1]  \times%
{\displaystyle \prod_{i\in I\backslash J}}
\left \{  0\right \}  \right)  =K.
\]
Hence $K$ is a continuous image of $%
{\displaystyle \prod_{i\in J}}
[0, 1]  ,$ which in turn a continuous image of $\left[
0,1\right]  .$ Thus $K$ is a continuous image of $\left[  0,1\right]  $.

$\left(  5\right)  \Longrightarrow \left(  6\right)  $. Suppose $K$ is
connected and locally path connected and $E$ is any closed subset of $\left[
0,1\right]  ,$ and $f:E\rightarrow K$ is continuous. For $x,y\in K,$ let
$\Delta \left(  x,y\right)  $ be the infimum of the diameters of every path in
$K$ from $x$ to $y$. We first note that%
\[
\lim_{d\left(  x,y\right)  \rightarrow0}\Delta \left(  x,y\right)  =0.
\]
If this is not true, then there is an $\varepsilon>0$ and sequences $\left \{
x_{n}\right \}  $ and $\left \{  y_{n}\right \}  $ in $K$ such that $d\left(
x_{n},y_{n}\right)  \rightarrow0$ and $\Delta \left(  x_{n},y_{n}\right)
\geq \varepsilon$ for every $n\in \mathbb{N}$. Since $K$ is compact we can
replace $\left \{  x_{n}\right \}  $ and $\left \{  y_{n}\right \}  $ with
subsequences that converge to $x$ and $y$, respectively. Since $d\left(
x,y\right)  =\lim_{n\rightarrow \infty}d\left(  x_{n},y_{n}\right)  =0$, we see
that $x=y.$ Since $X$ is locally connected, there is path connected
neighborhood $U$ of $x$ such that $U$ is contained in the ball centered at $x$
with radius $\varepsilon/3$. There must be an $n$ such that $x_{n},y_{n}\in
U,$ and there must be a path $\gamma$ in $U$ from $x_{n}$ to $y_{n}$. Since
$\gamma$ is in the ball centered at $x$ with radius $\varepsilon/3$, the
diameter of $\gamma$ is at most $2\varepsilon/3$, which implies $\Delta \left(
x_{n},y_{n}\right)  <\varepsilon$, a contradiction.

We can clearly add $0$ and $1$ to $E$ and extend $f$ so that it is still
continuous. Hence we can assume that $0,1\in E$. We can write
\[
\left[  0,1\right]  \backslash E=\bigcup_{n\in I}\left(  a_{n},b_{n}\right)
\]
where $\left \{  \left(  a_{n},b_{n}\right)  :n\in I\right \}  $ is a disjoint
set of open intervals with $I\subseteq \mathbb{N}$. For each $n\in I$ we chose
a path $\gamma_{n}:\left[  a_{n},b_{n}\right]  \rightarrow K$ from $f\left(
a_{n}\right)  $ to $f\left(  b_{n}\right)  $ so that the diameter of
$\gamma_{n}\ $is less than $2\Delta \left(  f\left(  a_{n}\right)  ,f\left(
b_{n}\right)  \right)  $.

$\left(  7\right)  \Longrightarrow \left(  3\right)  $. Suppose $\left(
7\right)  $ is true. Suppose $\mathcal{B}$ is a unital C*-algebra,
$\mathcal{M}\in \mathcal{Y}$ and $\pi:\mathcal{B}\rightarrow \mathcal{M}$ is a
unital surjective $\ast$-homomorphism. Let $\rho:C\left(  K\right)
\rightarrow \mathcal{M}$ be a unital $\ast$-homomorphism. Since $K$ is
metrizable, $C\left(  K\right)  $ is separable, and since $\mathcal{M}%
\in \mathcal{Y}$, there a countable commuting family $\left \{  p_{1}%
,p_{2},\ldots \right \}  $ of projections in $\mathcal{M}$ such that
$\mathcal{\rho}\left(  C\left(  K\right)  \right)  \mathcal{\subseteq C}%
^{\ast}\left(  p_{1},p_{2},\ldots \right)  $. Let $E$ be the maximal ideal
space of $\mathcal{C}^{\ast}\left(  p_{1},p_{2},\ldots \right)  $. Since
$\mathcal{C}^{\ast}\left(  p_{1},p_{2},\ldots \right)  $ is generated by
countable many projections, $E$ is a totally disconnected compact metric space
and is therefore homeomorphic to a subset of the Cantor set. Hence there is an
$a=a^{\ast}\in \mathcal{C}^{\ast}\left(  p_{1},p_{2},\ldots \right)  $ such that
$\mathcal{C}^{\ast}\left(  p_{1},p_{2},\ldots \right)  =C^{\ast}\left(
a\right)  $ and $\sigma \left(  a\right)  $ (homeomorphic to $E$) is a subset
of the Cantor set.  Let $\Gamma:C^{\ast}\left(  a\right)  \rightarrow C\left(
\sigma \left(  a\right)  \right)  $ be the Gelfand map. Then $\Gamma \circ
\rho:C\left(  K\right)  \rightarrow C\left(  \sigma \left(  a\right)  \right)
$ is a unital $\ast$-homomorphism, so there is a continuous function
$\psi:\sigma \left(  a\right)  \rightarrow K$ so that
\[
\Gamma \left(  \rho \left(  f\right)  \right)  =f\circ \psi
\]
for every $f\in C\left(  K\right)  .$ By applying $\Gamma^{-1}$, we get%
\[
\rho \left(  f\right)  =\left(  f\circ \psi \right)  \left(  a\right)
\]
for every $f\in C\left(  K\right)  .$

By $\left(  7\right)  $, we can assume (by extending) that $\psi:\left[
0,1\right]  \rightarrow K$.  We can find $A\in \mathcal{B}$ with $0\leq A\leq1$
such that $\pi \left(  A\right)  =a.$ We now define $\nu:C\left(  K\right)
\rightarrow \mathcal{B}$ by
\[
\nu \left(  f\right)  =\left(  f\circ \psi \right)  \left(  A\right)  .
\]
Then $\nu$ is a unital $\ast$-homomorphism and, for every $f\in C\left(
K\right)  $, we have%
\[
\pi \left(  \nu \left(  f\right)  \right)  =\pi \left(  \left(  f\circ
\psi \right)  \left(  A\right)  \right)  =\left(  f\circ \psi \right)  \left(
a\right)  =\rho \left(  f\right)  .
\]
Hence, $\rho=\pi \circ \nu$. This proves $\left(  3\right)  $.
\end{proof}

\bigskip






Let $K$ be a compact metric space and let $x_0$ be  any point in $K$. Let us denote by $C_{0}(K\setminus\{x_0\})$  the $C^*$-algebra of all continuous functions on $K$ vanishing at $x_0$.

\begin{theorem}
\label{MatricesOverCommAlgebras} Let $K$ be a continuous image of $\left[
0,1\right]  $, $n\in \mathbb{N}$. Then the $C^{*}$-algebra $C_{0}(K\setminus\{x_0\}) \otimes M_{n}$
is $C^{*}$-$W^{*}$-projective in the non-unital category.
\end{theorem}

\begin{proof}
Since $C_0(K\setminus\{x_0\})^{**}$ is a commutative von Neumann algebra, $$C_0(K\setminus\{x_0\})^{**} \cong C(X),$$ for some extremally disconnected space $X$.
Let $i_{\ast}: C_0(K\setminus\{x_0\}) \to C(X)$ be the canonical embedding into the bidual. It is induced by some surjective continuous map $i: X \twoheadrightarrow K$.
Since $K$ is a continuous image of $\left[  0,1\right]  $, there is  a surjective continuous map $\alpha: [0, 1]\twoheadrightarrow K$.
By the universal property of extremally disconnected spaces (Gleason's theorem), $i$ factorizes through a continuous map $\beta: X \to [0, 1]$.
$$\xymatrix {[0, 1] \ar@{->>}[d]^-{\alpha}  &  \\
K & X \ar@{->>}[l]_-{i}  \ar@/_/[ul]^{\beta} }$$
Let $$\alpha_{\ast}: C_0(K\setminus\{x_0\}) \to C_0(0, 1], \; \; \beta_{\ast}: C_0(0, 1]\to C_0(X) \subset C(X)$$ be the $\ast$-homomorphisms induced by $\alpha$ and $\beta$.
Let $\mathcal M$ be a von Neumann algebra, $\mathcal B$ a $C^*$-algebra, $q: \mathcal B\to \mathcal M$ a surjective $\ast$-homomorphism. Let $\phi: C_0(K\setminus\{x_0\}) \otimes M_n \to \mathcal M$ be a $\ast$-homomorphism. By the universal property of double duals we can extend $\phi$ to a $\ast$-homomorphism $\tilde \phi: C(X) \otimes M_n \cong (C_0(K\setminus\{x_0\})\otimes M_n)^{**} \to \mathcal M$.
$$\xymatrix { B \ar[dd]^{q} & & &\\  && C_0(0, 1]\otimes M_n  \ar[dr]_{\beta_{\ast}\otimes id_{M_n}} \ar@{.>}[ull]_{\psi}& \\ \mathcal M & C_0(K\setminus \!\{x_0\})\!\otimes \! M_n   \ar[l]_{\phi\;\;\;\;\;\;\;\;\;\;\;\;\;\;} \ar@{^{(}->}[rr]^{i_{\ast}\otimes id_{M_n}} \ar[ur]_{\alpha_{\ast}\otimes id_{M_n}} & & {C(X)\otimes M_n}
\ar@/^/[lll]^{\tilde \phi}
}$$
Since $C_0(0, 1]\otimes M_n$ is projective (\cite{LoringBook}, Th. 10.2.1), there is a $\ast$-homomorphism $\psi: C_0(0, 1]\otimes M_n \to \mathcal B$ such that
$$q\circ \psi = \tilde \phi \circ (\beta_{\ast}\otimes id_{M_n}).$$ Then $$q\circ \psi \circ (\alpha_{\ast}\otimes id_{M_n}) = \tilde \phi \circ (\beta_{\ast}\otimes id_{M_n}) \circ (\alpha_{\ast}\otimes id_{M_n}) =\tilde \phi \circ (i_{\ast}\otimes id_{M_n}) =  \phi.$$ Thus $\psi \circ (\alpha_{\ast}\otimes id_{M_n})$ is a lift of $\phi$.
\end{proof}

We don't know if the previous result can be generalized to  get the
following: if $A$ is $C^{*}$-$W^{*}$-projective in the non-unital category, then
so is $A\otimes M_{n}$.

\medskip

Recent developments in the classification of $C^{*}$-algebras show the
importance of the analysis of central sequence algebras.  Let
$\omega$ be a free ultrafilter on $\mathbb{N}$. With any $C^*$-algebra $A$ with a faithful tracial state $\tau$ one can associate its $C^{*}$-central sequence algebra $A_{\omega}\cap A^{\prime}$ and its $W^{*}$-central sequence algebra $N^{\omega}\cap
N^{\prime}$, where $N$ is the weak closure of $A$ under the
GNS representation $\pi_{\tau}$ of $A$. There is a natural $\ast$-homomorphism $\gamma: A_{\omega}\cap A^{\prime} \to N^{\omega}\cap
N^{\prime}$; it was proved in \cite{Sato} and \cite{KR} that $\gamma$ is surjective.

A commonly used tool in classification of $C^{*}$-algebras is the fact that an
order zero map from the matrix algebra $M_{n}$ to any quotient $C^{*}$-algebra
lifts (the so-called projectivity of order zero maps). In particular a
possibility to lift an order zero map $M_{n} \to N^{\omega}\cap N^{\prime}$ to
an order zero map $M_{n}\to A_{\omega}\cap A^{\prime}$ is a key ingredient to
obtain uniformly tracially large order zero maps (\cite{TWW}). Below we prove
a stronger statement: one can lift any commuting family of order zero maps
$M_{n} \to N^{\omega}\cap N^{\prime}$ to a commuting family of order zero maps $M_{n}\to A_{\omega}\cap A^{\prime}$.

\begin{lemma}
\label{universal} Let $d\in \mathbb{N}$, $k\in \mathbb{N }\cup \{ \infty \}$. The
$C^{*}$-algebra $C_{0}([0, 1]^{k}) \otimes M_{kd}$ is isomorphic to the
universal $C^{*}$-algebra with generators $e_{ij}^{n, l}$, $l\le k$, $i, j\le
d$, $n\in \mathbb{N}$ and relations
\[
0\le e_{ii}^{n, l} \le1
\]
\[
\left( e_{ij}^{n, l}\right) ^{\ast} = e_{ji}^{n, l},
\]
\[
e_{ij}^{m, l}e_{ks}^{n, l} = \delta_{jk}e_{is}^{m,l}, \; m< n
\]
\[
e_{ij}^{n, l}e_{ks}^{n, l} = \delta_{jk}e_{ii}^{n, l}e_{is}^{n, l},
\]
\[
[e_{ij}^{n, l}, e_{i^{\prime}j^{\prime}}^{n^{\prime}, l^{\prime}}] =
[e_{ij}^{n, l}, \left( e_{i^{\prime}j^{\prime}}^{n^{\prime}, l^{\prime}%
}\right) ^{*}] =0 \; \text{when} \; l\neq l^{\prime}.
\]

\end{lemma}

\noindent The proof of the lemma is analogous to the proof of Lemma 6.2.4 in
\cite{LoringBook}, so we don't write it.

\begin{theorem}\label{OrderZero}
Any commuting family of order zero maps from $M_{n}$ to $N^{\omega}\cap
N^{\prime}$ lifts to a commuting family of order zero maps from $M_{n}$ to
$A_{\omega}\cap A^{\prime}$.
\end{theorem}

\begin{proof}  As was proved in \cite{WZ}, with any  order zero map $\phi:B\to D$ one can associate a  $\ast$-homomorphism $f: CB\to D$ and vice versa. Here $CB = C_0(0, 1]\otimes B$ is the cone over $B$.
Moreover it follows from the construction in \cite{WZ} that order zero maps $\phi_i$'s have commuting ranges if and only if the corresponding $f_i$'s have commuting ranges. It also follows from the construction that if for an order zero map $\phi: A \to B/I$  the corresponding $f: CA \to B/I$ lifts to a $\ast$-homomorphism $\tilde f: CA \to B$, then
$\phi$ lifts to the order zero map $\tilde \phi$ corresponding to $\tilde f$. Thus we need to prove that
any family of $\ast$-homomorphisms $f_i: CM_n \to N^{\omega}\cap N'$ with pairwisely commuting ranges lifts to a family of $\ast$-homomorphisms $\tilde f_i:  CM_n \to A_{\omega}\cap A'$ with pairwisely commuting ranges.  It follows from Lemma \ref{universal} that it is equivalent to lifting of one $\ast$-homomorphism from $C_0([0, 1]^k) \otimes M_{kd}$ to $ N^{\omega}\cap N'$, where $k$ is the number of $\ast$-homomorphisms in the family. Since $[0, 1]^k$ is a continuous image of $[0, 1]$, the statement follows from Theorem \ref{MatricesOverCommAlgebras}.
\end{proof}

\section{$W^*$-$C^*$-projectivity and $RR0$-projectivity}

Recall that a $C^*$-algebra has {\it real rank zero } (RR0) if each its self-adjoint element can be approximated by self-adjoint elements with finite spectra.
Since a $\ast$-homomorphic image of a real rank zero C*-algebra is real rank
zero, and since every  von Neumann algebra has real rank zero \cite{BP}, it follows that every
$RR0$-projective C*-algebra is W*-C* projective.

 We first prove  that in the $W^{*}$-$C^{*}$-case we can lift projections. This was proved in the RR0-case by L. G. Brown and G. Pederson
\cite{BP}. We include the brief W*-C* proof because it is much shorter.

\begin{proposition}
\label{proj}$\mathbb{C}\oplus \mathbb{C}$ is  W*-C*-projective.
\end{proposition}

\begin{proof}
 Suppose $\mathcal{B}$ is a von Neumann algebra,
$\mathcal{M}$ is a C*-algebra, and $\pi:\mathcal{B}\rightarrow \mathcal{M}$ is
a unital $\ast$-homomorphism. Since $\mathbb{C}\oplus \mathbb{C}$ is the unital universal $C^*$-algebra of one projection, suppose $q\in \mathcal{M}$ is a projection. By [\cite{LoringBook}, Lemma 10.1.12], we can lift $q$ to $b_1$ and $1-q$ to $b_2$ in $\mathcal B$ such that $0\le b_j\le 1$,   $\pi(b_1) = q, \pi(b_2) = 1-q$ and  $b_{1}b_{2}=b_{2}b_{1}=0.$ Let
$p\in \mathcal{B}$ be the range projection for $b_{1}$. Then $pb_{2}=b_{2}p=0.$
Thus $b_{1}\leq p$ and $b_{2}\leq1-p$. Hence $q\leq \pi \left(  p\right)  $ and
$1-q\leq1-\pi \left(  p\right)  $. Hence, $\pi \left(  p\right)  =q$.
\end{proof}

\begin{definition}
Suppose $B,\mathcal{M}$ are unital C*-algebras and $\pi:\mathcal{B}%
\rightarrow \mathcal{M}$ is a surjective $\ast$-homomorphism, and suppose
$\mathcal{S}$ is a C*-subalgebra of $\mathcal{M}$. We say that $\gamma$\emph{
is a }$\ast$\emph{-cross section for }$\pi$\emph{ on }$S$ if $\gamma
:\mathcal{S}\rightarrow \mathcal{B}$ is a $\ast$-homomorphism and $\pi
\circ \gamma$ is the identity on $\mathcal{S}$. Clearly, every such $\gamma$ is injective.
\end{definition}

\begin{theorem}
\label{projlift}Suppose $\mathcal{B}$ is a real rank zero C*-algebra,
$\mathcal{M}$ is a C*-algebra, and $\pi:\mathcal{B}\rightarrow \mathcal{M}$ is
a unital surjective $\ast$-homomorphism. Suppose $\left \{  p_{1},p_{2}%
,\ldots \right \}  \subseteq \mathcal{M}$ is a commuting family of projections.
Then there is an unital $\ast$-cross section $\gamma$ for $\pi$ on $C^{\ast
}\left(  p_{1},p_{2},\ldots \right)  .$ Moreover, if $\gamma_{n}:C^{\ast
}\left(  p_{1},p_{2},\ldots,p_{n}\right)  \rightarrow \mathcal{M}$ is a unital
$\ast$-cross section for $\pi$ on $C^{\ast}\left(  p_{1},p_{2},\ldots
,p_{n}\right)  ,$ then there is a unital $\ast$-cross section $\gamma_{n+1}$
for $\pi$ on $C^{\ast}\left(  p_{1},p_{2},\ldots,p_{n+1}\right)  $ whose
restriction to $C^{\ast}\left(  p_{1},\ldots,p_{n}\right)  $ is $\gamma_{n}$.
\end{theorem}

\begin{proof}
We know that if $0\neq p\neq1$ is a projection in $\mathcal{M}$, there is a
projection $P\in \mathcal{B}$ with $\pi \left(  P\right)  =p.$ Clearly $0\neq
P\neq1$. Suppose $\gamma_{n}:C^{\ast}\left(  p_{1},p_{2},\ldots,p_{n}\right)
\rightarrow \mathcal{M}$ is a unital $\ast$-cross section for $\pi$ on
$C^{\ast}\left(  p_{1},p_{2},\ldots,p_{n}\right)  $. We know $C^{\ast}\left(
p_{1},p_{2},\ldots,p_{n}\right)  $ is generated by an orthogonal family of
projections $\left \{  q_{1},\ldots,q_{m}\right \}  $ whose sum is $1$, and, for
$1\leq k\leq m$, let $Q_{k}=\gamma_{n}\left(  q_{k}\right)  $. Now $p_{n+1}$
commutes with $\left \{  q_{1},\ldots,q_{m}\right \}  $, so $C^{\ast}\left(
p_{1},\ldots,p_{n+1}\right)  $ is generated by the orthogonal family
$$\cup_{k=1}^{m}\left \{  q_{k}p_{n+1}q_{k},q_{k}-q_{k}p_{n+1}q_{k}\right \}  .$$
If $q_{k}p_{n+1}q_{k}=0$ or $q_{k}-q_{k}p_{n+1}q_{k}=0$, there is nothing new
to lift. If $q_{k}p_{n+1}q_{k}\neq0$ and $q_{k}-q_{k}p_{n+1}q_{k}\neq0$, we
need to find a lifting of $q_{k}p_{n+1}q_{k}$ in $Q_{k}\mathcal{B}Q_{k}$.
However, if was proved in \cite{BP} that $Q_{k}\mathcal{B}Q_{k}$ has real rank $0$. Thus such a lifting is
possible.
\end{proof}

As a corollary we get a sufficient condition for $C(K)$ to be $RR0$-projective.

\begin{corollary}
\label{totallydisconnected} If $K$ is a totally disconnected compact metric
space, then $C\left(  K\right)  $ is RR0-projective, and hence W*-C* projective.
\end{corollary}

The following corollary uses the fact that if $\gamma:\mathcal{A}%
\rightarrow \mathcal{B}$ is a unital $\ast$-homomorphism and $\mathcal{B}$ is a
von Neumann algebra, then there is an extension to a weak*-weak* continuous
$\ast$-homomorphism $\hat{\gamma}:\mathcal{A}^{**}\rightarrow \mathcal{B}$.

\begin{corollary}
\label{commW*-C*proj} Let $K$ be a compact metric space. The following are equivalent.

\begin{enumerate}
\item $C\left(  K\right)  $ is W*-C* projective

\item Whenever $\mathcal{B}$ is a von Neumann algebra, $\mathcal{M}$ is a
C*-algebra, $\pi:\mathcal{B}\rightarrow \mathcal{M}$ is a surjective $\ast
$-homomorphism, and $\rho:C\left(  K\right)  \rightarrow \mathcal{M}$ is a
unital $\ast$-homomorphism, there is a commutative C*-subalgebra $\mathcal{D}$
of $\mathcal{M}$ that contains $\rho(C(K))$ such that the maximal ideal space
of $\mathcal{D}$ is totally disconnected,

\item Whenever $\mathcal{B}$ is a von Neumann algebra, $\mathcal{M}$ is a
C*-algebra, $\pi:\mathcal{B}\rightarrow \mathcal{M}$ is a surjective $\ast
$-homomorphism, and $\rho:C\left(  K\right)  \rightarrow \mathcal{M}$ is a
unital $\ast$-homomorphism, $\rho$ extends to a $\ast$-homomorphism $\hat
{\rho}:C\left(  K\right)  ^{**}\rightarrow \mathcal{M}$.
\end{enumerate}
\end{corollary}

\bigskip

\begin{remark}
Without the separablility assumption on $C\left(  K\right)  $ (i.e., the
metrizability of $K$), it is not generally true that $C\left(  K\right)  $ is
W*-C*-projective whenever $K$ is compact and totally disconnected. For
example, let $\mathcal{A}$ be the universal C*-algebra generated by a mutually
orthogonal family $\left \{  P_{t}:t\in \left[  0,1\right]  \right \}  $ of
projections. The maximal ideal space of $\mathcal{A}$ is the one-point
compactification $K$ of the discrete space $\left[  0,1\right]  $. If we let
$\mathcal{B}=B\left(  \ell^{2}\right)  $ and $\mathcal{M}=B\left(  \ell
^{2}\right)  /\mathcal{K}\left(  \ell^{2}\right)  $ and let $\pi
:\mathcal{B}\rightarrow \mathcal{M}$ be the quotient map, then there is an
injective unital $\ast$homomorphism $\rho:\mathcal{A}\rightarrow \mathcal{M}$,
but there is no injective unital $\ast$-homomorphism from $\mathcal{A}$ to
$\mathcal{B}$ since $B\left(  \ell^{2}\right)  $ does not contain an
uncountable orthogonal family of nonzero projections. Hence $C\left(
K\right)  $ is not W*-C* projective although $K$ is totally disconnected. This
shows that an attempt at a transfinite inductive version of the proof of
Theorem \ref{projlift} is doomed to failure. This also shows that being W*-C*
projective is not closed under arbitrary direct limits, since $\mathcal{A}$ is
the direct limit of the family $\left \{  C^{\ast}\left(  \left \{  P_{t}:t\in
E\right \}  \right)  :E\subseteq \left[  0,1\right]  \text{ is countable}%
\right \}  $. We doubt that $\rho$ can be extended to a $\ast$-homomorphism
from $C\left(  K\right)  ^{**}$ to $\mathcal{M}$, so the equivalence of
$\left(  1\right)  $ and $\left(  3\right)  $ in Corollary \ref{commW*-C*proj}
may conceivably be true.
\end{remark}

\medskip

In the case when $K$ is a Peano continuum (that is  a non-empty compact connected metric space which is locally connected at each point)  there is a necessary condition for $C(K)$ to be $W^*$-$C^*$-projective.

\begin{proposition}\label{Peano} Suppose $K$ is a Peano continuum and $C(K)$ is $W^*$-$C^*$-projective. Then $dim_{\text{cov}}(K) \le 1$ (here $dim_{\text{cov}}$ is the covering dimension).
\end{proposition}
\begin{proof} Suppose $dim_{\text{cov}} K  > 1$.  Then by [\cite{ChDr}, Prop. 3.1] $K$ contains a circle, $S^1$. Let $j: C(K) \to C(S^1)$ be the restriction map. Let $\pi: B(H) \to B(H)/K(H)$ be the canonical surjection.  Let $T\in B(H)$ be the unilateral shift. Define a $\ast$-homomorphism $\rho: C(S^1) \to B(H)/K(H)$ by sending the identity function $z$ to $\pi(T)$. We claim that $\rho\circ j$ is not liftable. Indeed suppose it lifts to a $\ast$-homomorphism $\gamma: C(K) \to B(H)$. Let $f\in C(K)$ be any preimage of $z\in C(S^1)$ under the map $j$. Since $f$ is normal, $\gamma(f)\in B(H)$ is a normal preimage of $\pi(T)$. Since any preimage of $\pi(T)$ has Fredholm index $-1$ and since any normal Fredholm operator has Fredholm index zero, we come to a contradiction.
\end{proof}

\medskip

Below we give a few non-commutative examples and non-examples of $W^*$-$C^*$-projective  $C^*$-algebras.
The following lemma shows that Murray-von Neumann equivalent projections can
be lifted to Murray von Neumann equivalent projections in the W*-C* case. More
simply, it states that partial isometries can be lifted.

\begin{lemma}
\label{PI} Let $\mathcal{B}$ be a unital C*-algebra, $\mathcal{M}$ be a von
Neumann algebra and  $\pi:\mathcal{M}\rightarrow \mathcal{B}$ be a surjective
$\ast$-homomorphism. Let $v\in \mathcal B$ be a partial isometry with $v^{\ast}v=p$ and
$vv^{\ast}=q$. Let $X\in \mathcal{M}$, $\left \Vert X\right \Vert \leq1$, and let
$P,Q$ be projections in $\mathcal{M}$ such that  $\pi \left(  P\right)
=p,\pi \left(  Q\right)  =q$ and $\pi \left(  X\right)  =v$. If $V$ is the
parial isometry in the polar decomposition of $PXQ$, then $\pi \left(
V\right)  =v$.
\end{lemma}

\begin{proof}
Let $Y=PXQ$. Then the range projection $P_{1}$ of $Y$ is less than or equal to
$P$ and the range projection of $Y^{\ast}$ is less than or equal to $Q$.
Moreover, $\left(  YY^{\ast}\right)  ^{1/2}\leq P_{1}\leq P$. Also,%
\[
p=\pi \left(  \left(  Y^{\ast}Y\right)  ^{1/2}\right)  \leq \pi \left(
P_{1}\right)  \leq \pi \left(  P\right)  =p\text{.}%
\]
If $Y=\left(  YY^{\ast}\right)  ^{1/2}V$ is the polar decomposition, then%
\[
v=p\pi \left(  V\right)  =\pi \left(  P_{1}V\right)  =\pi \left(  V\right)  .
\]
\end{proof}

\begin{corollary}\label{Toeplitz}
Let $\mathcal{T}$ be the Toeplitz algebra. Then $\mathcal{T }\oplus \mathbb{C}$
is W*-C* projective.
\end{corollary}

\begin{proof}
$\mathcal T \oplus \mathbb C$  is the universal unital $C^*$-algebra generated by $v$ with the relation that $v$ is a
partial isometry.
\end{proof}

\begin{corollary}
\label{M_n} $\mathcal{M}_{n}\left(  \mathbb{C}\right) \oplus \mathbb C$ is  W*-C* projective.
\end{corollary}

\begin{proof} It would be equivalent to prove that $\mathcal{M}_{n}\left(  \mathbb{C}\right)$ is W*-C* projective
in the non-unital category.
We will use induction on $n$. The case when $n=1$ amounts to lifting a projection.
Assume the theorem is true for $n$.
Suppose $\mathcal{M}$ is a
von Neumann algebra, $\mathcal{B}$ is a C*-algebra,  $\pi:\mathcal{M}\rightarrow \mathcal{B}$ is a surjective $\ast$-homomorphism.
Suppose $\rho:\mathcal{M}_{n}\left(  \mathbb{C}\right)  \rightarrow
\mathcal{B}$ is a $\ast$-homomorphism.
It follows from the induction assumption that there is a
$\ast$-homomorphism $$\gamma_{0}:C^{\ast}\left(  \left \{  e_{1k}:2\leq k\leq
n\right \}  \right)  \rightarrow \mathcal{M}$$ so that  $\gamma_{0}\left(
e_{jk}\right)  =E_{jk}\ $and $\left(  \pi \circ \gamma_{0}\right)  \left(
e_{1k}\right)  =\rho \left(  e_{1k}\right)  $ for $2\leq k\leq n$. We then have
$E_{1k}E_{1k}^{\ast}=\gamma_{0}\left(  e_{11}\right)  $ for $2\leq k\leq n$.
Choose $X\in \mathcal{M}$ so that $\pi \left(  X\right)  =\rho \left(
e_{1,n+1}\right)  .$ If we replace $X$ with $E_{11}X\left(  1-\sum_{1\leq
k\leq n}E_{kk}\right)  ,$ and  let $V$ be the partial isometry in the polar
decomposition of $E_{11}X\left(  1-\sum_{1\leq k\leq n}E_{kk}\right)  $, we
have from Lemma \ref{PI} that $\pi \left(  V\right)  =\rho \left(
e_{1,n+1}\right)  ,$ and $VV^{\ast}\leq E_{11}$ and $\pi \left(  VV^{\ast
}\right)  =\rho \left(  e_{11}\right)  $. If we replace $E_{1k}$ with
$F_{1k}=VV^{\ast}E_{1k}$ for $2\leq k\leq n$, and define $F_{1,n+1}=V$, we
obtain a representation $\gamma:\mathcal{M}_{n+1}\left(  \mathbb{C}\right)
\rightarrow \mathcal{M}$ with $\gamma \left(  e_{1k}\right)  =F_{1k}$ for $1\leq
k\leq n+1$ such that $\pi \circ \gamma=\rho$.
\end{proof}

\begin{lemma}
\label{directsums} Let $C^*$-algebras $A$  and $ D$ be unital $W^{*}$-$ C^{*}$-projective ($W^{*}$-$W^{*}$-projective respectively) in the non-unital category. Then $A \oplus D$
is $W^{*}$-$C^{*}$-projective ($W^{*}$-$W^{*}$-projective respectively) in the
non-unital category.
\end{lemma}

\begin{proof}   Suppose $\mathcal{B}$ and $\mathcal{M}$ are von Neumann algebras ($\mathcal{B}$ is a von Neumann algebra in the $W^*-W^*$-case respectively)  and
$\pi:\mathcal{B}\rightarrow \mathcal{M}$ is a surjective  $\ast
$-homomorphism. Let $\phi: A\oplus D \to \mathcal M$ be a $\ast$-homomorphism. Let $$p= \phi(1_A\oplus 0), \; q= \phi(0\oplus 1_D).$$ Define  $\phi_A: A\to pMp$ and $\phi_D: D \to qMq$ by $\phi_A(a) = \phi(a\oplus 0)$ and $\phi_D(d) = \phi(0\oplus d)$ respectively.  It follows from Corollary \ref{totallydisconnected} that we can lift $p, q$ to  projections $P, Q$ in $B$ which are orthogonal to each other. Let $\psi_A: A\to PBP$ and $\psi_D: D\to QBQ$ be lifts of $\phi_A$ and $\phi_D$.
Define lift $\psi$ of $\phi$ by $$\psi(a\oplus d) = \psi_A(a) + \psi_D(d).$$
\end{proof}

Combining this lemma with Corollary \ref{M_n} we obtain the following result.

\medskip

\begin{corollary}
\label{fd}If $\mathcal{A}$ is a finite-dimensional C*-algebra, then
$\mathcal{A}\oplus \mathbb{C}$ is W*-C* projective.
\end{corollary}

\begin{remark}
The result in Corollary \ref{fd} cannot be extended to AF-algebras even in the
W*-W* case. Indeed the tracial ultraproduct $%
{\displaystyle \prod \limits_{n\in \mathbb{N}}^{\alpha}}
\left(  \mathcal{M}_{2^{n}}\left(  \mathbb{C}\right)  ,\tau_{2^{n}}\right)  $
with respect to a free ultrafilter $\alpha$, where $\tau_{2^{n}}$ is the
normalized trace on $\mathcal{M}_{2^{n}}\left(  \mathbb{C}\right)  $, is a von
Neumann algebra. Thus $\pi:%
{\displaystyle \prod_{n\in \mathbb{N}}}
\mathcal{M}_{2^{n}}\left(  \mathbb{C}\right)  \rightarrow%
{\displaystyle \prod \limits_{n\in \mathbb{N}}^{\alpha}}
\left(  \mathcal{M}_{2^{n}}\left(  \mathbb{C}\right)  ,\tau_{2^{n}}\right)  $
is a unital surjective $\ast$-homomorphism and the domain and range are both
von Neumann algebras. If $\mathcal{A}$ is the CAR algebra, then it is clear
that there is an embedding $\rho:\mathcal{A}\rightarrow%
{\displaystyle \prod \limits_{n\in \mathbb{N}}^{\alpha}}
\left(  \mathcal{M}_{2^{n}}\left(  \mathbb{C}\right)  ,\tau_{2^{n}}\right)  $.
However, $\mathcal{A}$ is simple and infinite-dimensional, so there is no
embedding  from  $ \mathcal{A} \oplus \mathbb C$ into $%
{\displaystyle \prod_{n\in \mathbb{N}}}
\mathcal{M}_{2^{n}}\left(  \mathbb{C}\right)  $ such that $\rho=\pi \circ \tau$.
\end{remark}

A trace $\psi$ on a unital MF-algebra $\mathcal A$ is called an {\it MF-trace}  if  there is a free ultrafilter $\alpha$ on $\mathbb N$ and a unital
$\ast$-homomorphism $\pi: \mathcal A \to \prod^{\alpha} M_k(\mathbb C)$ to the $C^*$-ultraproduct of matrices, such that $\psi = \tau_{\alpha} \circ \pi$, where $\tau_{\alpha}(\{A_k\}_{\alpha} = \lim_{k\to \alpha} \tau_k(A_k)$  [\cite{all}, Prop.4].

The ideas in the preceding remark easily extend to the following result.

\begin{proposition}
If $\mathcal{A}$ is a unital MF C*-algebra and is  W*-C*-projective, then $\mathcal{A}$ must be RFD.
If the MF-traces are a faithful
set on $\mathcal{A}$, i.e., $\tau \left(  a^{\ast}a\right)  =0$ for every MF
trace implies $a=0$, and if $\mathcal{A}$ is W*-W* projective, then
$\mathcal{A}$ must be RFD.
\end{proposition}

The following result shows that without adding $\mathbb{C}$ as a direct
summand Corollaries \ref{Toeplitz} and \ref{M_n} no longer hold.

\begin{proposition}\label{M_nNonUnital}
$\mathcal{T}$ and $M_{n}(\mathbb{C})$ are not $W^{*}$-$C^{*}$-projective.
\end{proposition}

\begin{proof}
The Toeplitz algebra is not $W^*$-$C^*$-projective, since  an isometry in Calkin algebra need not lift to an isometry in $B(H)$.
$M_n(\mathbb C)$ is not  $W^*-C^*$-projective, because   $M_n(\mathbb C)$ is  a quotient of
$M_n(\mathbb C) \oplus \mathbb C $ and since  $M_n(\mathbb C)$ does not admit any unital $\ast$-homomorphisms to $\mathbb C$, the identity
map on  $M_n(\mathbb C)$ is not liftable.
\end{proof}

\begin{proposition} Suppose $\mathcal{A}$ is a separable unital C*-algebra.

\begin{enumerate}
\item If $Ext\left(  \mathcal{A}\right)  $ is not trivial, then $\mathcal{A}$
is not W*-C* projective

\item If $Ext_{w}\left(  \mathcal{A}\right)  $ is not trivial, then
$ \mathbb C \oplus \mathcal{A}$ is not W*-C* projective.
\end{enumerate}

\end{proposition}

\begin{proof} $\left(  1\right)  $ This is obvious.

$\left(  2\right)  $ Suppose $Ext_{w}\left(  \mathcal{A}\right)  $ is not
trivial. Then there is an injective unital $\ast$-homomorphism $\rho:\mathcal{A}%
  \rightarrow B\left(  \ell^{2}\right)
/\mathcal{K}\left(  \ell^{2}\right)  $ that is not weakly equivalent to the
trivial element in $Ext_{w}\left(  \mathcal{A}\right)  $. Assume, via
contradiction that there is a nonunital $\ast$-homomorphism $\gamma
:\mathcal{A\rightarrow B}\left(  \ell^{2}\right)  $ such that $\pi \circ
\gamma=\rho$. Then $$\pi \left(  1-\gamma \left(  1\right)  \right)
=1-\rho \left(  1\right)  =0.$$ Thus $1-\gamma \left(  1\right)  $ is a
finite-rank projection, and if $\gamma_{0}\left(  A\right)  =\gamma \left(
A\right)  |_{\gamma \left(  1\right)  \left(  \ell^{2}\right)  }$, we have
$\gamma=0\oplus \gamma_{0}$ relative to $\ell^{2}=\ker \gamma \left(  1\right)
\oplus \gamma \left(  1\right)  \left(  \ell^{2}\right)  .$ Since $\rho=\pi
\circ \gamma$ is injective, $\gamma_{0}$ must be injective. Choose an isometry
$V$ in $B\left(  \ell^{2}\right)  $ whose range is $\gamma \left(  1\right)
\left(  \ell^{2}\right)  .$ Then $V^{\ast}\gamma \left(  \cdot \right)  V$ is
unitarily equivalent to $\gamma_{0}$. Thus $\pi \left(  V\right)  $ is unitary
in $B\left(  \ell^{2}\right)  /\mathcal{K}\left(  \ell^{2}\right)  $ and
$\pi(V^{\ast})\rho \left(  \cdot \right)  \pi(V)$ lifts to $V^{\ast}\gamma \left(  \cdot \right)  V = U^{\ast}\gamma_{0}\left(
\cdot \right)  U$ for some unitary $U.$ This means $\rho$ is weakly equivalent
to the trivial element in $Ext_{w}\left(  \mathcal{A}\right)  $, a
contradiction.
\end{proof}

\begin{corollary}
If $n\geq2$, the Cuntz algebra $\mathcal{O}_{n}$ is not W*-C* projective. If
$n\geq3$, $\mathbb{C}\oplus O_{n}$ is not W*-C* projective.
\end{corollary}
\begin{proof} By [\cite{Davidson}, Th. V.6.5] $Ext(O_n) \cong \mathbb Z$ and by [\cite{Davidson}, Th. V. 6.6] $Ext_w(O_n) \cong \mathbb Z_{n-1}$, when $n\ge 2$.
\end{proof}

\begin{remark}
 By Corollary \ref{M_n} and Proposition \ref{M_nNonUnital}, if $n\geq2$, $\mathcal{M}_{n}\left(  \mathbb{C}\right)  $ is not W*-C*
projective, but $\mathbb{C}\oplus \mathcal{M}_{n}\left(  \mathbb{C}\right)  $
is W*-C* projective, and this happily coincides with the fact that
$Ext\left(  \mathcal{M}_{n}\left(  \mathbb{C}\right)  \right)  $ is not
trivial and $Ext_{w}\left(  \mathcal{M}_{n}\left(  \mathbb{C}\right)  \right)
$ is trivial.
\end{remark}

The following is a consequence of the proof of a result of T. Loring and the
second-named author [\cite{LoringShulmanAlgebraicElements}, Th. 9]. It
generalizes Olsen's structure theorem for polynomially
compact operators \cite{Olsen}.

\begin{theorem}
Let $R\geq0$ and $p\in \mathbb{C}\left[  x\right]  $. The universal C*-algebra
generated by $a$ such that $\left \Vert a\right \Vert \leq R$ and $p\left(
a\right)  =0$ is $RR0$-projective and hence $W^{*}$-$C^{*}$-projective.
\end{theorem}

\section{$W^*$-$W^*$-projectivity}

We begin with the separable unital commutative C*-algebras.

\begin{theorem}
\label{commutative} Every separable unital commutative C*-algebra is RR0-AW*- projective. In particular every separable unital commutative C*-algebra is W*-W*-projective.
\end{theorem}

\begin{proof}
Suppose $\mathcal{B}$ is a real rank zero C*-algebra,  $\mathcal{M}$ is an
AW*-algebra and
$\pi:\mathcal{B}\rightarrow \mathcal{M}$ is a surjective unital $\ast
$-homomorphism, and suppose that $\mathcal{A}$ is a separable unital commutative C*-subalgebra
of $\mathcal{M}$. Since $\mathcal{M}$ is an AW*-algebra, every maximal abelian
selfadjoint C*-subalgebra of $\mathcal{M}$ is the C*-algebra generated by its
projections. Since $\mathcal{A}$ is contained in such a maximal algebra and
$\mathcal{A}$ is separable, it follows that there is a countable commuting
family $\left \{  p_{1},p_{2},\ldots \right \}  $ of projections in $\mathcal{M}$
such that $\mathcal{A}\subset C^{\ast}\left(  p_{1},p_{2},\ldots \right)  $. By
Theorem \ref{projlift} there is a $\ast$-cross section $\gamma$ for $\pi$ on
$C^{\ast}\left(  p_{1},p_{2},\ldots \right)  .$ Clearly, the restriction of $g$
to $\mathcal{A}$ is a $\ast$-cross section of $\pi$ for $\mathcal{A}$.
\end{proof}

\begin{theorem}
\label{matrices over algebras} Let $A$ be a unital $C^{*}$-algebra. If $A$
is $W^{*}$-$C^{*}$-projective ($W^{*}$-$W^{*}$-projective respectively) in the
non-unital category, then  for each $n\in \mathbb{N}$, $M_{n}(A)$ is
$W^{*}$-$C^{*}$-projective ($W^{*}$-$W^{*}$-projective respectively) in the
non-unital category.
\end{theorem}

\begin{proof} Our proof is a modification of Loring's proof of the
fact that the class of projective $C^*$-algebras is closed under tensoring with matrices
(\cite{LoringBook}).
Let $\phi: M_n\otimes A \to B/I$ be a $\ast$-homomorphism and $B$
(and $B/I$, for the $W^*$-$W^*$-projectivity case) be  a von Neumann
algebra and let $\pi: B\to B/I$ denote the canonical surjection. We
need to prove that $\phi$ lifts. Define $j: M_n \to M_n\otimes A$ by
$$j(T) = T\otimes 1_A$$ and let $\phi_2 = \phi \circ j$. Since by
Corollary \ref{M_n} $M_n$ is $W^*$-$C^*$-projective in the non-unital category, $\phi_2$ lifts to $\psi:
M_n\to B$.
$$\xymatrix {M_n \ar[r]^-{\psi}\ar[d]^-{j} \ar@/_/[dr]^{\phi_2} & B\ar[d]^-{\pi} \\
M_n\otimes A \ar[r]_-{\phi}& B/I}$$ Let $(e_{ij})$ be a matrix unit
in $M_n$. Define a $\ast$-homomorphism $$i: M_n\otimes
\phi(e_{11}\otimes A) \to \phi(M_n\otimes A)$$ by $i(T\otimes
\phi(e_{11}\otimes a)) = \phi(T\otimes a).$ It is obviously
surjective. To see that it is injective, we will use the fact that
an ideal in a tensor product $C^*$-algebra is a tensor product of
ideals. Hence the kernel of $i$ is either $0$ or of the form
$M_n\otimes J$, where $J$ is an ideal in $\phi(e_{11}\otimes A)$.
Let $\phi(e_{11}\otimes a))\in J$. Then for each $T\in M_n$,
$\phi(T\otimes a) = T\otimes  \phi(e_{11}\otimes a) =0$. In
particular, $\phi(e_{11}\otimes a) = 0$. Thus $\phi(e_{11}\otimes
a))=0$ and $J=0$. So $i$ is an isomorphism. Let $$p= \phi_2(e_{11}),
\;P = \psi(e_{11})$$ and let $i_1$ be the inclusion
$\phi(e_{11}\otimes A)\subseteq pB/Ip= PBP/PIP.$ Then the
composition $(id_{M_n}\otimes i_1)\circ i^{-1}\circ \phi: M_n\otimes A \to M_n\otimes
PBP/PIP$ is of the form $id_{M_n}\otimes \gamma$, where $\gamma:
A \to  PBP/PIP$ is defined by $$\gamma(a) =
p\phi(e_{11}\otimes a)p.$$ Since $PBP$ (and $pB/Ip$, for the
$W^*$-$W^*$ -projectivity case) is a von Neumann algebra,  by
$W^*$-$C^*$ ($W^*$-$W^*$) projectivity of $A$, it can be lifted to
$$id_{M_n}\otimes \tilde \gamma: M_n\otimes A \to M_n\otimes PBP.$$
$$\xymatrix { & & & M_n\otimes PBP \ar[d]^{id_{M_n}\otimes \pi|_{PBP}} \\
M_n\otimes A \ar[urrr]^-{id_{M_n}\otimes \tilde
\gamma}\ar[r]_-{\phi}  & \phi(M_n \otimes A) \ar[r]_{i^{-1}} &
M_n\otimes \phi(e_{11}\otimes A) \ar[r]_{id_{M_n}\otimes i_1} & M_n\otimes pB/Ip}$$
Now we are going to  embed $M_n\otimes PBP$ back into $B$ and $ M_n\otimes pB/Ip$ -- back into $B/I$. Define $$\tilde
\alpha: M_n\otimes PBP \to B \; \; \text{and} \; \; \alpha: M_n\otimes p
B/I p \to B/I$$ by $$\tilde \alpha(e_{ij}\otimes PbP) =
\psi(e_{i1})b\psi(e_{1j})$$ and $$\alpha(e_{ij}\otimes p\pi(b)p) =
\phi_2(e_{i1})\pi(b)\phi_2(e_{1j})$$ respectively, for each $b\in
B$. It is straightforward to check that $$\alpha \circ (id_{M_n}\otimes i_1) \circ
i^{-1} \circ \phi  = \phi$$ and that the diagram
$$\xymatrix {M_n\otimes PBP \ar[r]^-{\tilde \alpha}\ar[d]^-{id_{M_n}\otimes \pi|_{PBP}} & B\ar[d]^-{\pi} \\
M_n\otimes pB/Ip \ar[r]_-{\alpha}& B/I}$$ commutes. It follows that
$\tilde \alpha \circ \left(id_{M_n}\otimes \tilde \gamma \right)$ is
a lift of $\phi$.
\end{proof}

\begin{remark}
We did not consider the $C^{*}$-$W^{*}$ case in the theorem because no unital
$C^{*}$-algebra is $C^{*}$-$W^{*}$-projective in the non-unital category.
Otherwise $A\oplus \mathbb{C}$ would be unital and $C^{*}$-$W^{*}$-projective
which would contradict to Proposition \ref{NoProjections} since $A\oplus
\mathbb{C}$ has a non-trivial projection.
\end{remark}

Recall that a $C^*$-algebra is {\it subhomogeneous} if there
is an upper bound for the dimensions of its irreducible representations.

\begin{theorem}
\label{subhomogeneous} Let $A$ be a separable  subhomogeneous $C^{*}$-algebra. Then $A$
is $W^*$- $W^{*}$-projective in the non-unital category.
\end{theorem}

\begin{proof}
Suppose $\mathcal{B}$ and $\mathcal{M}$ are von Neumann algebras and
$\pi:\mathcal{B}\rightarrow \mathcal{M}$ is a surjective  $\ast
$-homomorphism. Let $\phi: A \to \mathcal M$ be a $\ast$-homomorphism. If $A$ is non-unital, we can extend $\phi$ to a homomorphism from the unitization of $A$ to $\mathcal M$. It implies that it will be sufficient to prove the theorem in the assumption that $A$ is unital.
Since $\mathcal M \subseteq B(H)$, by the universal property of the second dual
there exists $\tilde \phi: A^{**}\to B(H)$ such that $\tilde \phi|_{A} = \phi$ and $\tilde \pi(A^{**}) = \pi(A)''.$ Hence $\tilde \phi$ is a $\ast$-homomorphism from $A^{**}$ to $M$. It can be easily deduced from some well-known properties of subhomogeneous algebras (see for instance \cite {TanyaOtogo}, Lemmas 2.3 and  2.4) that $A^{**}$ can be written as $$A^{**} = \oplus_{k=1}^n M_k(D_k),$$ where $D_k, k\le n,$ are abelian von Neumann algebras.   Let $\pi_k: A^{**}\to M_k(D_k)$ be the projection on the k-th summand. Let $$F_k = \{b\in D_k\;|\; \exists a\in A \; \text{such that b is a matrix element
of}\; \pi_k(a)\}, $$ for each $k\le n$. Let $E_k$ denote the $C^*$-subalgebra of $D_k$ generated by $F_k$, for each $k\le n$. Then each $E_k$ is separable and $A\subseteq \oplus_{k=1}^n M_k(E_k) \subseteq A^{**}$.
By Theorem \ref{commutative}, Theorem \ref{matrices over algebras} and Lemma \ref{directsums}, $\tilde \phi|_{\oplus_{k=1}^n M_k(E_k)}$ lifts to some $\ast$-homomorphism $\psi: \oplus_{k=1}^n M_k(E_k) \to \mathcal B$.  The restriction of $\psi$ onto $A$ is a lift of $\phi$.
\end{proof}

The following are easy observations.

\begin{proposition}
Suppose $\mathcal{A}$ is a separable unital $W^*$-$W^*$- projective C*-algebra.

1) If $\mathcal{A}$ is nuclear and has a faithful trace, then it must be RFD.

2) If Connes' embedding problem has an affirmative answer, then every unital
$W^*$-$W^*$-projective $C^*$-algebra with a faithful trace is RFD.
\end{proposition}

The opposite to the previous proposition is not true. Indeed in
\cite{TracialStability} we constructed a nuclear RFD $C^{*}$-algebra which is
not matricially tracially stable and hence is not $W^{*}-W^{*}$-projective.
Below we give an example which is not only nuclear but even AF. Our arguments
of why it is not matricially tracially stable are much simpler than the ones
in \cite{TracialStability}.

\begin{theorem}\label{AF}
There exists an AF RFD $C^{*}$-algebra which is not matricially tracially
stable and hence is not $W^{*}$-$W^{*}$-projective (in both unital and
non-unital categories).
\end{theorem}

\begin{proof}
Suppose $\mathcal{A}$ and $\mathcal{B}$ are separable unital AF-C*-algebras.
Suppose $\mathcal{A}=C^{\ast}\left(  a_{1},a_{2},\ldots \right)  $ and
$\mathcal{B=}C^{\ast}\left(  b_{1},b_{2},\ldots \right)  $ with each $a_{n}$
and $b_{n}$ selfadjoint. We can assume that $\sigma \left(  a_{1}\right)
\subset \left[  0,1\right]  $ and $\sigma \left(  b_{1}\right)  \subset \left[
4,5\right]  $. Then we can find, for each $n\in \mathbb{N}$, a
finite-dimesnional C*-subalgebra $\mathcal{A}_{n}$ of $\mathcal{A}$ and
elements $a_{1,n},\ldots,a_{n,n}\in \mathcal{A}_{n}$ such that $\left \Vert
a_{k}-a_{k,n}\right \Vert <1/n$ for $1\leq k\leq n$. Similarly, we can find,
for each $n\in \mathbb{N}$ a finite-dimensional C*-subalgebra $\mathcal{B}_{n}$
of $\mathcal{B}$ and elements $b_{1,n},\ldots,b_{n,n}\in \mathcal{B}_{n}$ such
that $\left \Vert b_{k}-b_{k,n}\right \Vert <1/n$ for $1\leq k\leq n$. We can
also assume that $\sigma \left(  a_{1,n}\right)  \subset \left[  -1,2\right]  $
and $\sigma \left(  b_{1,n}\right)  \subset \left[  3,6\right]  $ for every
$n\in \mathbb{N}$. We can assume, for each $n\in \mathbb{N}$ that $\mathcal{A}%
_{n},\mathcal{B}_{n}\subset \mathcal{M}_{s_{n}}\left(  \mathbb{C}\right)  $
(unital embeddings).  For each $1\leq k\leq n<\infty$ define%
\[
c_{k,n}=a_{k,n}^{\left(  n\right)  }\oplus b_{k,n}\in \mathcal{M}_{\left(
n+1\right)  s_{n}}\left(  \mathbb{C}\right)  .
\]
Define $c_{k,n}=0$ when $1\leq n<k<\infty$.
Let $C_{k}=\sum_{n\in \mathbb{N}}^{\oplus}c_{k,n}\in%
{\displaystyle \prod \limits_{n\in \mathbb{N}}}
\mathcal{M}_{\left(  n+1\right)  s_{n}}\left(  \mathbb{C}\right)  $ and define
the C*-algebra $C$ as the C*-algebra generated by $C_{1},C_{2},\ldots$ and
$\mathcal{J}=\sum_{n\in \mathbb{N}}^{\oplus}\mathcal{M}_{\left(  n+1\right)
s_{n}}\left(  \mathbb{C}\right)  $. Clearly, $\mathcal{C}$ is RFD and
$$\mathcal{C}/\mathcal{J} \cong C^{\ast}\left(  a_{1}\oplus
b_{1},a_{2}\oplus b_{2},...\right)  \subseteq \mathcal{A}\oplus \mathcal{B}.$$
However, if $f:\mathbb{R}\rightarrow \mathbb{R}$ is continuous and $f=0$ on
$\left[  0,1\right]  $ and $f=1$ on $\left[  2,3\right]  $, we have $f\left(
a_{1}\oplus b_{1}\right)  =0\oplus1.$ Thus $0\oplus1 \in C^{\ast}\left(  a_{1}\oplus
b_{1},a_{2}\oplus b_{2},...\right)$ and hence $C^{\ast}\left(  a_{1}\oplus
b_{1},a_{2}\oplus b_{2},...\right)  =\mathcal{A}\oplus \mathcal{B}$. Since
$\mathcal{J}$ and $\mathcal{C}/\mathcal{J}$ are AF, $\mathcal{C}$ must be AF.
Now, to get an example we wanted,  suppose $\mathcal{A}=\mathcal{B}=\mathcal{M}_{2^{\infty}}$ with
trace $\tau$. Let $\pi=\pi_{1}\oplus \pi_{2}:\mathcal{C}\rightarrow
\mathcal{A}\oplus \mathcal{B}$ be the map whose kernel is $\mathcal{J}$. Then
$\rho=\tau \circ \pi_{2}$ is a tracial state on $\mathcal{C}$. Note that since $\mathcal{J}\subset \mathcal C$, the
only irreducible finite-dimensional representations of $\mathcal{C}$ are
(unitarily equivalent to) the coordinate projections onto $\mathcal{M}%
_{\left(  n+1\right)  s_{n}}\left(  \mathbb{C}\right)  $ and, for each of
these representations the trace of the image of $f\left(  C_{1}\right)
=\sum_{n\in \mathbb{N}}^{\oplus}0^{\left(  n\right)  }\oplus1$ is at most
$1/2.$ However, $\rho \left(  f\left(  C_{1}\right)  \right)  =1$. Thus $\rho$
is not a weak*-limit of finite-dimensional traces. By [\cite{TracialStability}, Th. 3.10] $\mathcal C$ is not matricially tracially stable.
\end{proof}

\end{document}